\documentclass[11pt, twoside]{article}
\usepackage{amsfonts}
\usepackage{ulem}

\usepackage{amssymb}
\usepackage{amsmath}
\usepackage{amsthm}
\usepackage{xcolor}
\usepackage{mathrsfs}

\usepackage{verbatim}

\allowdisplaybreaks

\allowdisplaybreaks

\def\z+{{\mathbb Z}_+}

\def\supp{{\rm{\,supp\,}}}

\def\bint{{\ifinner\rlap{\bf\kern.30em--}
\int\else\rlap{\bf\kern.35em--}\int\fi}\ignorespaces}

\def\sbint{{\ifinner\rlap{\bf\kern.32em--}
\hspace{0.078cm}\int\else\rlap{\bf\kern.45em--}\int\fi}\ignorespaces}

\newtheorem{theorem}{Theorem}[section]
\newtheorem{lemma}[theorem]{Lemma}

\newtheorem{proposition}[theorem]{Proposition}
\theoremstyle{definition}

\newtheorem{remark}[theorem]{Remark}

\newtheorem{definition}[theorem]{Definition}
\numberwithin{equation}{section}

\numberwithin{equation}{section}

\numberwithin{equation}{section}

\begin{document}

\arraycolsep=1pt

\title{\Large\bf A new type of $bmo$ space for non-doubling measures \footnotetext{\hspace{-0.35cm} {\it 2020
Mathematics Subject Classification}. {42B35, 43A99.}
\endgraf{\it Key words and phrases.} non-doubling measure, $rbmo(\mu)$, $RBMO(\mu)$, cube of generation.
\endgraf This project is supported by the National Natural Science Foundation of China (Grant Nos. 12261083, 12161083) and Xinjiang Key Laboratory of Applied Mathematics (No. XJDX1401).
\endgraf $^\ast$\,Corresponding author.
}}
\author{Shining Li, Haijing Zhao and Baode Li$^\ast$}
\date{ }
\maketitle

\vspace{-0.8cm}

\begin{center}
\begin{minipage}{13cm}\small
{\noindent{\bf Abstract.}
Let $\mu$ be a Radon measure on $\mathbb R^{d}$ which may be non-doubling and only satisfies $\mu(Q(x,l))\le C_{0}l^{n}$}  for all $x\in \mathbb R^{d}$, $l(Q)>0$, with some fixed constants $C_{0}>0$ and $n\in (0,d]$. We introduce a new type of $bmo(\mu)$ space  which looks bigger than the  $rbmo(\mu)$ space of Dachun Yang (JAMS,\,2005).  And its four equivalent norms are established by constructing some special types of auxiliary doubling cubes. Then we further obtain that this new $rbmo(\mu)$ space   actually coincides with the  $rbmo(\mu)$ space of Dachun Yang.
\end{minipage}
\end{center}


\section{Introduction\label{s1}}

Let $d\in \mathbb N$  and $0<n \le d $. We assume that   $\mu$ is a positive Radon measure on $\mathbb R^{d}$ satisfying the $growth \ condition$
\begin{equation}\label{eq1.1}
\mu(Q(x,l))\le C_{0}l^{n}  \quad    \mathrm{for\ all}\ x \in \mathbb R^{d} \  \mathrm{and} \ l>0 \  \mathrm{ with \ some \ fixed \ constant} \ C_{0}>0.
\end{equation}
We don't assume that $\mu$ is doubling and the cube in the measure can also be replaced by a ball.

It is well known that the doubling condition on Radon measure plays an important role in most results of classical function spaces, Harmonic Analysis and so on. However,  in 2002,  it has been shown that many results in the classical Calder\'on-Zygmund  theory and the classical
$Hardy$ and $BMO$ spaces also hold without the doubling assumption  and only need the measure to satisfy the growth condition in \cite{4,6}.

In 1961,  John and Nirenberg \cite{X1} introduced the $BMO$ space. The space $BMO$ shares similar properties with the space
$L^{\infty}$, and it often serves as a substitute for it. For instance, classical singular integrals do not map $L^{\infty}$ to $L^{\infty}$ but $L^{\infty}$ to $BMO$. And in many instances interpolation
between $L^{p}$ and $BMO$ works just as well between $L^{p}$ and $L^{\infty}$. In 1979, Goldberg \cite{1} introduced a local $BMO$-type space $bmo$ which is the dual of the local $Hardy$ space $h^{1}$.
\begin{definition}\label{d1.7}
\begin{itemize}
\item[\rm(i)]\cite[p.15]{1}
A function $f\in L_{loc}(\mathbb R^{d})$ is said to be in the space $bmo(\mathbb R^{d})$ if there exists a nonnegative constant $C_{a}$ such that, for any cube $Q$ with $l(Q)< 1$,
\begin{equation}\label{eq1.7}
 \frac{1}{|Q|} \int_{Q}\bigg|f(y)-\frac{1}{|Q|} \int_{Q}f(x)dx \bigg|dy \le C_{a}
\end{equation}
and that, for any cube $Q$ with $l(Q)\in [1,\infty)$,
\begin{equation}\label{eq1.8}
\frac{1}{|Q|} \int_{Q}|f(y)|dy\le C_{a}  .
\end{equation}
Then, we write $\|f\|_{bmo(\mathbb R^{d})}=\inf C_{a}$, where the infimum is taken over all the constants $C_{a}$.
\item[\rm(ii)]\cite[p.2]{6} If there exists some constant $C_{b}>0$  such that (\ref{eq1.7}) holds for $C_{b}$ instead of $C_{a}$, (\ref{eq1.8}) holds for any cube $Q$ with $l(Q)=1$ and $C_{b}$ instead of $C_{a}$. Then, the best constants $C_{b}$ is comparable to the $bmo(\mathbb R^{d})$ norm of $f$.
\item[\rm(iii)]If there exists some constant $C_{c}>0$ and a fixed constant $k>1$ such that (\ref{eq1.7}) holds for $C_{c}$ instead of $C_{a}$, (\ref{eq1.8}) holds for any cube $Q$ with $l(Q)\in[1,k]$ and $C_{c}$ instead of $C_{a}$. Then, by (i) and (ii), it is easy to check that the best constants $C_{c}$ is comparable to the $bmo(\mathbb R^{d})$ norm of $f$.
\end{itemize}
\end{definition}
 From Definition \ref{d1.7}, it implies that the set of functions in $bmo(\mathbb R^{d})$ doesn't depend on the choice of the value range of $l(Q)$ in (\ref{eq1.8}).

In 2000,  Mateu, Mattila, Nicolau and Orobitg  \cite{3} wanted to find a good substitute for the space $BMO$ when
the underlying measure is non-doubling and  hoped to prove some results via $BMO$-$H^{1}$ interpolation, but in this respect they were unsuccessful. Then, in 2001, Tolsa \cite{4} found a suitable substitute for the classical $BMO$ space in this setting, which is denoted by $RBMO(\mu)$. This space is small enough to posses the properties
such as the John-Nirenberg inequality and big enough for Calder\'on-Zygmund operators which are bounded on $L^{2}(\mu)$ to be also bounded from $L^{\infty}(\mu)$ into $RBMO(\mu)$ (see \cite[p.2]{XX}). The  $(\alpha, \beta)$-doubling cubes and the coefficients $K_{Q,R}$ introduced by Tolsa play  important roles for their $RBMO$ theory. Precisely,

\begin{definition}\cite[p.6]{4}\label{d1.2}
Let $0<n\le d$, $\alpha >1$ and $\beta >\alpha^{n}$.  We say that some cube $ Q \subset \mathbb R^{d}$ is $(\alpha,\beta)$-$doubling$ if $\mu(\alpha Q) \le \beta \mu(Q)$.
\end{definition}

In \cite[p.8]{4}, for any cube $Q\subset \mathbb R^{d}$, Tolsa introduced the smallest expansive concentric $(\alpha,\beta)$-doubling cube of $Q$, denoted as $\widetilde Q:=\alpha^{m}Q$ with some $m\in \mathbb N$, and the existence of such cube is guaranteed by Proposition \ref{p1.3}(i) below.

\begin{definition}\cite[p.6]{4}\label{d1.4}
Give two cubes $Q,R\subset \mathbb R^{d}$ with $Q\subset R $. We denote
$$K_{Q,R}:=1+\sum_{k=1}^{N_{Q,R}}\frac{\mu(2^{k}Q)}{l(2^{k}Q)^{n}} ,  $$
where $N_{Q,R}$ is the first integer $k$ such that $l(2^{k}Q)\ge l(R)$ (in case $R=\mathbb R^{d} \not= Q$, we set $K_{Q,R}=\infty$).
\end{definition}

\begin{definition}\cite[p.8]{4}\label{d1.1}
 Let $0<n\le d$,   $ 1<\eta,\, \alpha <\infty$ and $\beta>\alpha^{n}$. We say that $f\in L_{loc}(\mu ) $ is in $RBMO(\mu)$ if there exists some constant $C_{1}\ge 0$ such that, for any cube $ Q$,
 \begin{equation}\label{eq1.2}
\frac{1}{\mu(\eta Q)} \int _{Q}\left | f(x)-m_{\widetilde{Q}}f  \right | d\mu(x) \le C_{1},
\end{equation}
and   for any two $(\alpha,\beta)$-doubling cubes $Q\subset R $,
\begin{equation}\label{eq1.3}
| m_{Q}f-m_{R}f   |  \le C_{1}K_{Q,R}  ,
\end{equation}
where $m_{Q}f$ stands for the mean of $f$ over $Q$ with respect to $\mu$, i.e., $$m_{Q}f:=\frac{1}{\mu(Q)}\int_{Q}f(x)d\mu(x).$$

The infimum  constant $C_{1}$ is the $RBMO(\mu)$ norm of $f$, and  denoted  by $||f||_{RBMO(\mu)}$.
\end{definition}

In 2005, Dachun Yang \cite{5} considered the local versions of the space $RBMO(\mu)$, i.e.,  $rbmo(\mu)$  in the sense of Goldberg \cite{1} by using ideas coming from \cite{4}.
\begin{definition}\label{d1.6}\cite[Definition 2.1(ii)]{5}
Let $ 1< \eta \le \alpha < \infty $, $0<n\le d $ and $\beta \ge \alpha^{n} $. We say that $f\in L_{loc}(\mu ) $ is in $rbmo(\mu)$ if there exists some constant $C_{2}\ge 0$ such that, for any cube $ Q$ with $l(Q)\le1$,
   \begin{equation}\label{eq1.4}
\frac{1}{\mu(\eta Q)} \int _{Q}\left | f(x)-m_{\widetilde{Q}} f \right | d\mu(x) \le C_{2},
\end{equation}
and that, for any two $(\alpha ,\beta  ) $-doubling cubes $Q\subset R $ with $l(Q)\le 1$,
\begin{equation}\label{eq1.5}
| m_{Q}f-m_{R}f   |  \le C_{2}K_{Q,R}  ,
\end{equation}
and that, for any cube $Q$ with $l(Q)>1$,
    \begin{equation}\label{eq1.6}
\frac{1}{\mu(\eta Q)} \int _{Q}\left | f(x)  \right | d\mu(x) \le C_{2}.
\end{equation}
We define the $rbmo(\mu)$ norm of $f$ by the infimum constant $C_{2}$ and we denote this by $\|f\|_{rbmo(\mu)}$.
\end{definition}
 Furthermore,  Yang \cite{5} gave some basic properties including several equivalent definitions of this local space. By using these properties, he established the John-Nirenberg's inequality for the functions in the  $rbmo(\mu)$ space and a $rbmo(\mu)$-$h^{1}(\mu)$ interpolation of  linear operators.

In \cite{5}, the $rbmo(\mu)$ space  introduced by  Yang is the non-doubling version of Definition \ref{d1.7}(i).
So, naturally, two questions come out that  whether the non-doubling versions of (ii) and (iii) of Definition \ref{d1.7} are equivalent with $rbmo(\mu)$ of Yang. In this article, we have made some progress for Definition \ref{d1.7}(iii) but still open for Definition \ref{d1.7}(ii). Precisely, for the definition of the $rbmo(\mu)$ space, if we only replace the condition $l(Q)>1$ of (\ref{eq1.6})  to $1<l(Q)\le k$, we still don't know how to prove the equivalent of them. Alternatively, if the condition $l(Q)>1$ of (\ref{eq1.6}) is replaced by $\mathscr Q_{ex} \setminus \mathscr Q$, where
$$ \mathscr{Q}:=\left \{Q:l(Q)\le 1\right \}, \  \mathscr{Q}_{ex}:=\left \{Q:l(Q')\le 1\right \},$$
$Q'$ denotes the biggest contractive concentric   $(\alpha,\beta)$-doubling cube of a cube $Q\subset \mathbb R^{d}$, i.e., $Q':=\alpha^{-N} Q$ with some $N\in \mathbb Z_{+}$, and  the existence of such  $(\alpha,\beta)$-doubling cube is guaranteed  by Proposition \ref{p1.3}(ii).
Then the equivalence of this new $rbmo(\mu)$ space and $rbmo(\mu)$ space of Yang can be verified in this article. The new space is denoted as $rbmo^{1}_{\mathscr Q}(\mu)$, which is a non-doubling version of Definition \ref{d1.7}(iii). Here we remark that $\mathscr Q_{ex} \setminus \mathscr Q=\left\{Q\subset \mathbb R^{d}:1<l(Q)\le \alpha\right\}$ when $\mu$ is doubling.

\begin{definition}\label{d1x.1}
Let $ \eta \in (1,\infty)$ and $ \alpha \in [\eta,\infty)$,  $ \mathscr{Q}:=\left \{Q:l(Q)\le 1\right \}$ and $ \mathscr{Q}_{ex}:=\left \{Q:l(Q')\le 1\right \}$. We say that $f\in L_{loc}(\mu ) $ is in the $rbmo^{1}_{\mathscr Q}(\mu)$ space if there exists some constant $C_{3}\ge 0$ such that, for any cube $ Q\in \mathscr{Q}$,
\begin{equation}\label{eq2.1}
\frac{1}{\mu(\eta Q)} \int _{Q}\left | f(x)-m_{\widetilde{Q} }f  \right | d\mu(x) \le C_{3},
\end{equation}
and for any two $(\alpha ,\beta  ) $-doubling cubes $Q\subset R $ with $Q\in \mathscr{Q}$ and $R\in \mathscr{Q}_{ex} $,
\begin{equation}\label{eq2.2}
| m_{Q}f-m_{R}f   |  \le C_{3}K_{Q,R}  ,
\end{equation}
and for any cube  $Q\in \mathscr{Q}_{ex} \setminus  \mathscr{Q}$,
\begin{equation}\label{eq2.3}
\frac{1}{\mu(\eta Q)} \int _{Q}\left | f(x)  \right | d\mu(x) \le C_{3} .
\end{equation}
We write $ \|f\|_{\ast1}= \inf C_{3} $, where the infimum  is taken over all the constants $C_{3}$  satisfying (\ref{eq2.1}), (\ref{eq2.2}) and (\ref{eq2.3}).

\end{definition}

From the definitions of $rbmo(\mu)$ and $rbmo^{1}_{\mathscr Q}(\mu)$, it is easy to see that $rbmo(\mu) \subset rbmo^{1}_{\mathscr Q}(\mu)$. By using the  Besicovich  covering lemma and some subtle estimates related to $(\alpha,\beta)$-doubling cubes, we further obtain that they  are actually the same space with equivalent norms which is the following Theorem \ref{t3.1}. Therefore,  the $rbmo^{1}_{\mathscr Q}(\mu)$ space  inherits all the properties of $rbmo(\mu)$ of Yang such as John-Nirenberg's inequality (see \cite[Theorem 2.6]{5})  for the functions in $rbmo(\mu)$ space and that the dual space of local Hardy space $h_{1}(\mu)$ is $rbmo(\mu)$ space (see \cite[Lemma 4.2.17]{6}), and so on.

 The organization of this article is as follows.  In Section 2, we obtain some results for $(\alpha,\beta)$-doubling cubes such as $(\alpha,\beta)$-doubling cubes $Q'$ and $\widetilde Q$ of a cube $Q\subset \mathbb R^{d}$ and two families of cubes $\mathscr Q_{ex}$ and $\mathscr Q$.
 In Section 3, four non-doubling versions of Definition \ref{d1.7}(iii) are introduced and their equivalence are also verified.
 Then Section 4 is devoted to prove $rbmo(\mu)=rbmo^{1}_{\mathscr Q}$ and $rbmo(\mu) \subset RBMO(\mu)$.

\section*{Notation}
\ \,\ \ $-\mathbb Z:$ the set of integers;

 $-\mathbb Z_{+}:=\left\{1,2,,3,\dots\right\}$;

 $-\mathbb N := \left\{0,1,2,\dots \right\}$;

$-\mu:$ a positive Radon measure on $\mathbb R^{d}$;

$-\supp(\mu):=
(\bigcup\limits_{\mu(E_{\alpha})=0}E_{\alpha})^{\complement}$;

$-\l(Q):$ the side length  of $Q$;

$-\ x_{Q}:$ the center  of $Q$;

$-\eta Q:$ the cube concentric with $Q$ with side length $\eta l(Q)$;

$-m_{Q}f:=\frac{1}{\mu(Q)}\int_{Q}f(x)d\mu(x)$;

$-\left\{f_{Q}\right\}_{Q}:$ a collection of numbers that is, for each cube $Q$ there exists $f_{Q}\in \mathbb R$ ;

$-C:$ a positive constant which is independent of the main parameters, but it may vary from line to line. Constant with subscripts, such as $C_{1}$ and $C_{2}$, do not change in different occurrences.




\section{Some results for $(\alpha,\beta)$-doubling cubes \label{s2}}

In this section, we will prove the existence of the two special $(\alpha,\beta)$-doubling cubes $\widetilde Q$ and $Q'$ for a cube $Q\subset \mathbb R^{d}$.
Then some basic properties of $\mathscr Q_{ex}$ and $\mathscr Q$ will be given.
Let us begin with the Besicovich covering lemma, which is very useful and important in our context.
\begin{lemma}(\cite[Theorem 1.1.1]{6})\label{ll2.1}
Let $E$ be a bouned set in $\mathbb R^{d}$. If, for every $x\in E$, there exists a closed cube $Q_{x}$ centered at $x$, then it is possible to choose, from among the given cubes $\left\{Q_{x}\right\}_{x\in E}$, a subsequence $\left\{Q_{k}\right\}_{k}$ (possibly finite) such that
\begin{itemize}
\item[\rm(i)] $E\subset \bigcup\limits_{k}Q_{k};$
\item[\rm(ii)] no point of $\mathbb R^{d}$ is in more than $C_{d}$ (a number that only depends on d) cubes of the sequences $\left\{Q_{k}\right\}_{k}$, namely, for every $z\in \mathbb R^{d}$,
    $$\sum_{k}\chi_{Q_{k}}(z)\le C_{d}.$$
\end{itemize}
\end{lemma}

In Proposition \ref{p1.3}, we  obtain that there exist a lot of big $(\alpha,\beta)$-doubling cubes by (i) and many small $(\alpha,\beta)$-doubling cubes by (ii). The proof of Proposition \ref{p1.3} is inspired by \cite[Proposition 1.2.2]{6}.
\begin{proposition}\label{p1.3}
Let $0<n\le d$, $\alpha >1$ and $\beta >\alpha^{n}$.
\begin{itemize}
\item[\rm(i)] For $\mu$-almost every $x\in \mathbb R^{d}$ and $c\in(0,\infty)$, there exists some $(\alpha,\beta)$-doubling cube $Q$ centered at $x$ with $l(Q)=\alpha^{N}c$ and $N\in \mathbb N$.
\item[\rm(ii)] If $\beta > \alpha^{d} $, for $\mu$-almost every $x\in \mathbb R^{d}$ and $c\in(0,\infty)$, there exists some $(\alpha,\beta)$-doubling cube $Q$ centered at $x$ with $l(Q)=\alpha^{-N}c$ and $N\in \mathbb Z_{+}$.
\end{itemize}
\end{proposition}
\begin{proof}
To prove (i), we assume that (i) doesn't hold true, i.e., there exist some $x_{0}\in \supp(\mu)$ and positive constant $M\in \mathbb N$ such that, for all cubes $Q$ centered at $x_{0}$ with $l(Q)=\alpha^{N}c$ and  the integer $N\ge M$, we have $\mu(\alpha Q)>\beta \mu(Q)$.

Now let us take $Q_{0}$ be such a cube with $\mu(Q_{0})>0$.  By our assumption and the growth condition, we see that, for any $k\in \mathbb N$,
$$\beta^{k}\mu(Q_{0})<\beta^{k-1}\mu(\alpha Q_{0})< \cdots<\beta \mu(\alpha^{k-1}Q_{0})<\mu(\alpha^{k}Q_{0})\le C_{0}\alpha^{kn}[l(Q_{0})]^{n},$$
which in turn implies that
$$\mu(Q_{0})<C_{0}\left(\frac{\alpha^{n}}{\beta}\right)^{k}[l(Q_{0})]^{n}.$$
Letting $k\to \infty$, we have $\mu(Q_{0})=0$, which contracts to $\mu(Q_{0})>0$. This implies that, for any $x_{0}\in \supp(\mu)$, there exists some $(\alpha,\beta)$-doubling cube $Q$ centered at $x_{0}$ with $l(Q)=\alpha^{N}c$, i.e., (i) holds true.

Then we prove (ii). For any fixed $c>0$, $\alpha\in (1,\infty)$ and $\beta \in (\alpha^{d},\infty)$, let
\begin{equation*}
\begin{aligned}
\Omega :=\left\{{x\in \mathbb R^{d}:\mathrm {there \ doesn't \ exist \ any \ (\alpha,\beta)-doubling \ cubes  }}\right.\\
\left. { \mathrm { \ centered \   at } \ x  \ \mathrm { whose\ side \ lengths \ are \ } \alpha^{-N}c \mathrm{ \ for \ all }\ N \in \mathbb Z_{+} }\right\}.
\end{aligned}
\end{equation*}
We show that $\mu(\Omega)=0$. For any $m\in \mathbb Z_{+}$, let
\begin{equation*}
\begin{aligned}
\Omega_{m} :=\left\{{x\in \mathbb R^{d}:\mathrm {all \ cubes \ centered \ at   }\ x \mathrm{\ whose \ side \ lengths \ are \ } \ \alpha^{-N}c}\right.\\
\left. { \mathrm {\ are \ not \ (\alpha,\beta)-doubing \ cubes }  \mathrm{ \ for \ all }\ N \ge m \mathrm {\ with} \ N\in \mathbb Z_{+}} \right\}.
\end{aligned}
\end{equation*}
Observe that $\Omega=\cup^{\infty}_{m=1}\Omega_{m}$. It suffices to prove that $\mu(\Omega_{m})=0$ for any $m\in \mathbb Z_{+}$.

Let us fix a cube $Q$ with $l(Q)=\alpha^{-(m+1)}c$ and denote by $Q^{n}_{x}$ the cube centered at $x$ whose side length is $\alpha^{-n}l(Q) $ for any $x\in \Omega_{m} \cap Q$ and $n\in \mathbb Z_{+}$. By the Besicovitch covering lemma (Lemma \ref{ll2.1}), there exists a sequence of cubes, $\left\{Q^{n}_{k}\right\}_{k\in I_{n}}$, such that
$$\Omega_{m}\cap Q \subset \bigcup_{k\in I_{n}}Q^{n}_{k} .$$
Since the center of $Q^{n}_{k}$ is in $\Omega$ and $l(Q^{n}_{k})=\alpha^{-(m+n+1)}c$, $Q^{n}_{k}$ is not a $(\alpha,\beta)$-doubling cube for each $k$. Therefore, from this and the fact that $\alpha^{n}Q^{n}_{k}\subset 3Q$, it follows that
\begin{equation}\label{2.xx1}
\mu(Q^{k}_{n})< \beta^{-1}\mu(\alpha Q^{k}_{n}) < \dots <\beta^{-n}\mu(\alpha^{n} Q^{k}_{n})\le \beta^{-n}\mu(3Q)
\end{equation}
Then by \cite[(1.2.3)]{6},  we have  $\# (I_{n})\lesssim \alpha^{nd}$. By this and (\ref{2.xx1}), we obtain
$$\mu(\Omega_{m}\cap Q)\le \sum_{k\in I_{n}}\mu(Q^{n}_{k})\lesssim \alpha^{nd}\beta^{-n}\mu(3Q). $$
Letting $n\rightarrow \infty$, we see that $\mu(\Omega_{m}\cap Q)=0$ by $\beta >\alpha^{d}$.

Notice that, for each $m\in \mathbb N$, $\mathbb R^{d}=\cup_{i}Q_{m,i}$, where $\left\{Q_{m,i}\right\}_{i}$ are cubes with
$$l(Q_{m,i})=\alpha^{-(m+1)}c \qquad \mathrm{for\ all}\ i.$$
Then we have
$$\mu(\Omega_{m})\le \sum_{i}\mu(\Omega_{m}\cap Q_{m,i})=0.$$
This further implies that $\mu(\Omega)=0$.
\end{proof}
{\bf Basic Assumption:} In order to guarantee the existence of $Q'$ and $\widetilde Q $ for any cube $Q\subset \mathbb R^{d}$, by Proposition \ref{p1.3}, we always assume that $\alpha>1$ and $\beta>\alpha^{d}$ through the whole article.

\begin{remark}\label{r2.5}

Given a cube $ Q\subset \mathbb R^{d}$, we can pick a cube-chain $\left \{Q_{N}:= \alpha^{N}Q \right \}_{N\in \mathbb Z} $. By Proposition \ref{p1.3}, we know that there are infinite $ (\alpha, \beta)$-doubling cubes in the cube-chain $\left \{Q_{N} \right \}_{N\in \mathbb Z} $.
\begin{itemize}
   \item[\rm(i)] For any  $Q_{n}\in \left \{Q_{N} \right \}_{N\in \mathbb Z}$, $Q_{n}^{'}$ and $\widetilde Q_{n}$ are two $ (\alpha, \beta)$-doubling cubes beside   $Q_{n}$.
$\widetilde Q_{n}$ is the smallest  $ (\alpha, \beta)$-doubling cube with side length greater than $l(Q_{n}^{'} )$ in the cube-chain,  $ Q_{n}'$ is the biggest   $ (\alpha, \beta)$-doubling cube with side length less than $l(\widetilde Q_{n} )$ in the cube-chain. So there are no doubling cubes  between $ Q_{n}'$ and $\widetilde Q_{n}$ in the cube-chain. And for any cube $R $ between $ Q_{n}'$ and $\widetilde Q_{n}$ including $\widetilde Q_{n}$, we have $R'=Q_{n}'$ and $\widetilde R=\widetilde Q_{n}$.

  \item[\rm(ii)] For any $N\in \mathbb Z_{+} $, if $l((\alpha^{N}Q)')>l(Q')$  or $l(\widetilde {\alpha^{N}Q })>l(\widetilde Q)$,  then by (i),  we have $ \widetilde Q \subset (\alpha^{N}Q )'$.
\end{itemize}
\end{remark}

The following proposition was proved by Tolsa in \cite[Lemma 2.1]{4}, which plays a fundamental role in this whole article. From this proposition, it is easy to see that $K_{Q,R}$ reflects some geometric aspects of cubes and has a constant upper bound estimate in many cases.
\begin{proposition}\label{p1.5}
Let $C>0 $ be a constant.
\begin{itemize}
\item[\rm(i)] If $Q\subset R \subset S$ are cubes in $\mathbb R^{d}$, then $K_{Q,R}\le K_{Q,S}$, $K_{R,S}\le CK_{Q,S}$ and $K_{Q,S}\le C(K_{Q,R}+K_{R,S})$.
\item[\rm(ii)] If $Q\subset R$ have comparable sizes, then $K_{Q,R}\le C$.
\item[\rm(iii)] If $N$ is some positive integer and the cubes $\alpha Q,\alpha^{2}Q,\dots , \alpha ^{N-1}Q$ are non $(\alpha,\beta)$-doubling cube with $\beta >\alpha^{n}$, then $K_{Q,\alpha^{N}Q}\le C$ with C depending on $\beta$, $\rho$ and $C_{0}$.
\end{itemize}
\end{proposition}

We know that $\mathscr Q_{ex} =\left\{Q\subset \mathbb R^{d}:l(Q)\le \alpha\right\}$ when $\mu$ is doubling. However,
if $\mu$ is a non-doubling measure, the family of cubes $\mathscr Q_{ex}$ becomes  complex. Some basic  properties are given in the following proposition for $\mathscr Q_{ex}$, the proof of which is easy to check and hence  omitted.

\begin{proposition}\label{p2.4}
Let $\alpha>1$, $\mathscr{Q}:=\left \{Q:l(Q)\le 1\right \}, \  \mathscr{Q}_{ex}:=\left \{Q:l(Q')\le 1\right \}. $
\begin{itemize}
  \item[\rm(i)] For any $Q\in \mathscr Q_{ex}$, there exist a positive integer $N_{1}$ and a nonnegative integer $N_{2} $ such that $Q'=\alpha^{-N_{1}} Q $ and $\widetilde{Q}=\alpha^{N_{2}} Q $. Moreover,   $\alpha^{-N_{1}+1} Q $, $ \alpha^{-N_{1}+2} Q$, $ \dots $, $Q$, $\alpha Q$, $\dots$, $\alpha^{N_{2}} Q \in \mathscr Q_{ex}$.
\item[\rm(ii)]   If $Q \in \mathscr{Q}_{ex} $, then $\widetilde{Q} \in \mathscr Q_{ex}$.
\item[\rm(iii)] If $Q \in \mathscr Q$, then $\widetilde{\alpha Q}  \in \mathscr Q_{ex} $.
\item[\rm(iv)] If $1< l(Q) \le \alpha $, then $Q \in \mathscr Q_{ex} \setminus  \mathscr{Q} $.
\end{itemize}
\end{proposition}

\section{Equivalence of four $rbmo(\mu)$-type spaces \label{s3}}
  The definition of $rbmo^{1}_{\mathscr Q}(\mu)$ (Definition \ref{d1x.1}) is a non-doubling version of Definition \ref{d1.7}(iii), and in this section, we will further  introduce the other three non-doubling versions and prove  the equivalence of them with $rbmo^{1}_{\mathscr Q}(\mu)$, which are given by Theorems \ref{l2.10} and  \ref{l2.12}.

\begin{definition}\label{d2.2}
Let $\eta\in(1,\infty)$, $ \mathscr{Q}:=\left \{Q:l(Q)\le 1\right \}$ and $ \mathscr{Q}_{ex}:=\left \{Q:l(Q')\le 1\right \}$. Suppose that for a given function $f\in L_{loc}(\mu)$, we say that $f$ is in $rbmo^{2}_{\mathscr Q}(\mu)$ if there exist some constant $C_{4}\ge 0 $ and a collection of numbers $ \left \{ f_{Q} \right \}_{Q}$ (that is, for each cube $Q$ there exists $f_{Q}\in \mathbb R$) such that
\begin{equation}\label{eq2.4}
\sup_{Q\in \mathscr{Q}} \frac{1}{\mu(\eta Q)} \int _{Q}\left | f(x)-f_{Q}   \right | d\mu(x) \le C_{4},
\end{equation}
and that,  for any two  cubes $Q\subset R $ with $Q\in \mathscr{Q}$ and $R\in \mathscr{Q}_{ex} $,
\begin{equation}\label{eq2.5}
| f_{Q}-f_{R}   |  \le C_{4}K_{Q,R}  ,
\end{equation}
and that, for any cube  $R\in \mathscr{Q}_{ex} \setminus  \mathscr{Q}$,
\begin{equation}\label{eq2.6}
|f_{R}|\le C_{4} .
\end{equation}
Then, we write $ \left \| f  \right \| _{\ast2}= \inf\ {C_{4}}   $, where the infimum  is taken over all the constants $C_{4}$ and all the numbers $\left \{ f_{Q} \right \} $ satisfying (\ref{eq2.4}), (\ref{eq2.5}) and (\ref{eq2.6}).
\end{definition}
It is easily
checked that $\|f\|_{\ast1}$ and $\left \| f  \right \| _{\ast2} $ are the norms in the space of functions modulo constants.

The  $RBMO(\mu)$ space was introduced by Tolsa     and it was proved that the definition of $RBMO(\mu) $ is independent of the choice of $ \eta$  (see \cite[Lemmas 2.6 and  2.8]{4}). The space $rbmo(\mu)$  introduced by Dachun Yang  is also independent of the choice of $ \eta$  (see \cite[Propositions 2.3 and  2.4]{5}). Next, we will also verify  that the definition of $rbmo_{\mathscr Q}(\mu) $ is independent of the choice of $ \eta$ . Let us begin with two lemmas.

\begin{lemma}\label{l2.6}
Let $\rho>0$ and $ Q $ be a fixed cube in $ \mathbb R^{d} $. For any  cube $Q_{x}$ centered at $ x \in Q \cap \supp(\mu) $, if $ l(Q_{x}) \le \rho l(Q) $, then $Q_{x} \subset   (\rho+1) Q$.
\end{lemma}
\begin{proof}
Let $d=2$,  the proof for other values of $ d $ is similar. Let $ l_{x}(ab) $ and  $ l_{y}(ab) $ be the  projection lengths  of segment  $ ab $ on the  $x$-axis and  $y$-axis, $x_{Q}$ be the center of $ Q $ and  $x$ be the center of $ Q_{x} $. Then we obtain
$ l_{x}(x x_{Q} ) \le \frac{1 }{2} l(Q)  $, $ l_{y}(x x_{Q}) \le \frac{1 }{2} l(Q) $.

For any $ x' \in Q_{x} $, it is easy to find that $$ l_{x}(x' x) \le \frac{1 }{2} l(Q_{x}) \le  \frac{\rho}{2} l(Q) , l_{y}(x' x) \le \frac{1 }{2} l(Q_{x}) \le  \frac{\rho}{2} l(Q) .$$
So  we have
$$  l_{x}(x'x_{Q}) \le l_{x}(x'x)+l_{x}(xx_{Q}) \le \frac{1 }{2} l(Q) + \frac{\rho}{2} l(Q)\le \frac{1}{2}l((\rho+1)Q), $$
$$  l_{y}(x'x_{Q}) \le l_{y}(x'x)+l_{x}(xx_{Q}) \le \frac{1 }{2} l(Q) + \frac{\rho}{2} l(Q)\le \frac{1}{2}l((\rho+1)Q), $$
which implies  $x' \in (\rho+1) Q$, i.e., $Q_{x} \subset (\rho +1)Q  $.
\end{proof}

\begin{lemma}\label{l2.8}
Let $ f\in rbmo^{2}_{\mathscr Q}(\mu)$ and  $\left \{ f_{Q} \right \}_{Q}$ be a collection of numbers satisfying (\ref{eq2.4}),(\ref{eq2.5}) and (\ref{eq2.6}).  If $Q$ and $R$ are cubes in  $\mathscr{Q}_{ex}$ such that $l(Q)\approx l(R) $, $\min(l(Q),l(R))\le 1$ and $x_{Q} \in R $ $($$x_{Q} $ is the center of the cube Q$)$, then
$$ |f_{Q}-f_{R}| \le C \|f\|_{\ast2}.$$
\end{lemma}
\begin{proof}
Let us prove Lemma \ref{l2.8} by two cases.

Case 1.  If $l(Q) \le l(R)$, then we have $l(Q) \le 1 $ and hence $Q \in \mathscr{Q} $. Because the center of $Q$ is in $R $, then there exists  a cube $S$ such that $l(S)=\frac{1}{2}l(Q)$ and $S\subset Q\cap R$. From this we can conclude that $l(S)\approx l(Q)\approx  l(R) $, $ S\subset Q$, $ S\subset R$ and $S \in \mathscr Q $. By (\ref{eq2.5}) and Proposition \ref{p1.5}(ii), we have
\begin{align*}
|f_{Q}-f_{R}|\le |f_{Q}-f_{S}|+|f_{S}-f_{R}|\le (K_{S,Q}+K_{S,R})\|f\|_{\ast2 }\le C\|f\|_{\ast2 }.
\end{align*}

Case 2. If $l(Q) > l(R)$, by $ x_{Q} \in R$, we can find that $ x_{R}\in Q$, $l(R) \le 1 $ and $R \in \mathscr{Q} $. By this and using a similar  estimate of (i), we further obtain $|f_{R}-f_{Q}|\le \|f\|_{\ast2}$.

\end{proof}
The following proposition indicates that $\|f\|_{\ast2 (\eta)}$ is independent of $\eta$, the proof is a slight variation of \cite[Lemma 2.6]{4} with $|f_{Q_{xi}}-f_{Q_{0}}|\le C\|f\|_{\ast\ast(\eta)}$ replaced by Lemma \ref{l2.8}. The details are omitted.
\begin{proposition}\label{p2.9}
The norms $ \left \| \cdot  \right \| _{\ast2,(\eta)}$ are independent of $\eta \in (1,\infty) $.
\end{proposition}

\begin{theorem}\label{l2.10}
Let $\alpha >1$ and $\eta \in (1,\alpha] $. Then $rbmo^{1}_{\mathscr Q}(\mu)=rbmo^{2}_{\mathscr Q}(\mu)$ with equivalent norms.
\end{theorem}
\begin{proof}
 To prove  $\|f\|_{\ast2} \le C \|f\|_{\ast1}$ for $f\in rbmo^{1}_{\mathscr Q}(\mu)$, we set

$$ f_{Q}:=\begin{cases} m_{\widetilde{Q} }f    &    \widetilde{Q} \in \mathscr{Q};\\ 0     &        \widetilde{Q} \in \mathscr{Q}_{ex} \setminus  \mathscr{Q}.\end{cases}$$
Let us prove this by three steps.

Step 1.  Prove
\begin{equation}\label{eq2x.1}
 \underset {Q \in \mathscr{Q}}{\sup}\frac{1}{\mu(\eta Q)} \int _{Q} |f(x)-f_{Q}|d\mu(x)\le C\|f\|_{\ast1}  .
\end{equation}

If $\widetilde{Q} \in \mathscr{Q}$, obviously,  (\ref{eq2x.1}) holds with $C=1$.

If $\widetilde{Q} \in \mathscr{Q}_{ex} \setminus  \mathscr{Q}$, then by $\eta \le \alpha$ and that the cube $\widetilde Q$ is $(\alpha,\beta)$-doubling, we have
\begin{align*}
\frac{1}{\mu(\eta Q)} \int _{Q} |f(x)-f_{Q}|d\mu(x) =& \frac{1}{\mu(\eta Q)} \int _{Q} |f(x)|d\mu(x)    \\
 \le & \frac{1}{\mu(\eta Q)} \int _{Q} |f(x)-m_{\widetilde{Q}}f|d\mu(x)+\frac{\mu(Q)}{\mu(\eta Q)} | m_{\widetilde{Q}}f | \\
 \le & \|f\|_{\ast1}+ \frac{\mu(\eta \widetilde{Q} )}{\mu (\widetilde{Q})}  \|f\|_{\ast1} \le \|f\|_{\ast1}+ \frac{\mu(\alpha \widetilde{Q} )}{\mu (\widetilde{Q})}  \|f\|_{\ast1} \le C\|f\|_{\ast1}  .
\end{align*}
which implies that (\ref{eq2x.1}) holds.

Step 2.  Prove
    \begin{equation}\label{eq2.7}
|f_{Q}-f_{R}| \le C K_{Q,R} \|f\|_{\ast1}
\end{equation}
for any two cubes $Q\subset R$ with $Q\in \mathscr Q$ and $R \in \mathscr Q_{ex}$. We prove this by six cases.

Case 1.  $ l(\widetilde{R}) \ge l(\widetilde{Q})>1 $. (\ref{eq2.7}) is obviously true.

Case 2.  $l(\widetilde{R})> 1 \ge l(\widetilde{Q}) $.  By (\ref{eq2.2}), (\ref{eq2.3}), $\eta\le \alpha$, Remark \ref{r2.5}(i), (i) and (iii) of Proposition \ref{p1.5} and that cube $\widetilde R$ is $(\alpha,\beta)$-doubling, we obtain
\begin{align*}
|f_{Q}-f_{R}| =& |f_{Q}| = |m_{\widetilde Q}f| \\
\le &  |m_{\widetilde Q}f - m_{Q'}f |+|m_{Q'}f - m_{\widetilde R}f |+| m_{\widetilde R}f|\\
\le & K_{Q',\widetilde Q}\|f\|_{\ast1} + K_{Q',\widetilde R}\|f\|_{\ast1} + \frac{\mu(\eta \widetilde{R} )}{\mu (\widetilde{R})}  \|f\|_{\ast1} \\
\le & C\|f\|_{\ast1}+C(K_{Q',Q}+K_{Q, R}+K_{R,\widetilde R})\|f\|_{\ast1}+C \|f\|_{\ast1} \le C K_{Q, R} \|f\|_{\ast1},
\end{align*}
which implies that (\ref{eq2.7}) holds.

Case 3.   $ l(\widetilde{Q}) \le l(\widetilde{R}) \le 1 $.  By (\ref{eq2.2}), Remark \ref{r2.5}(i), (i)  and (iii) of Proposition \ref{p1.5}, we have
\begin{align*}
|f_{Q}-f_{R}| =&   |m_{\widetilde Q}f - m_{\widetilde R}f | \\
\le &  |m_{\widetilde Q}f - m_{Q'}f |+|m_{Q'}f - m_{\widetilde R}f |\\
\le &  K_{Q',\widetilde Q}\|f\|_{\ast1}+K_{Q',\widetilde R}\|f\|_{\ast1}\\
\le & C\|f\|_{\ast1}+C(K_{Q',Q}+K_{Q, R}+K_{R,\widetilde R})\|f\|_{\ast1}\\
\le & C\|f\|_{\ast1}+CK_{Q, R}\|f\|_{\ast1} \le C K_{Q, R} \|f\|_{\ast1} ,
\end{align*}
which implies that (\ref{eq2.7}) holds.

Case 4. $ 1 \le l(\widetilde{R}) \le l(\widetilde{Q}) $. (\ref{eq2.7}) is obviously true.

Case 5. $ l(\widetilde{R}) \le 1 < l(\widetilde{Q})   $. This case is similar to Case 2, and its proof is omitted.

Case 6. $ l(\widetilde{R}) \le l(\widetilde{Q}) \le 1 $.  This case is similar to Case 3, and its proof is omitted.

Thus (\ref{eq2.7}) holds in all cases.

Step 3.  Prove  $|f_{R}| \le C\|f\|_{\ast1}$ for any cube  $R\in \mathscr{Q}_{ex} \setminus  \mathscr{Q}$.

It is obviously true since $|f_{R}|=0\le C\|f\|_{\ast1}$ with $R\in \mathscr{Q}_{ex} \setminus  \mathscr{Q}$.

Combining Steps 1-3, we finally obtain that $  \|f\|_{\ast2}  \le C\|f\|_{\ast1}$.

Let us now prove $  \|f\|_{\ast1 }  \le C\|f\|_{\ast2}$ by three steps.

Step I. Prove
  \begin{equation}\label{eq2.8}
 \frac{1}{\mu(\eta Q)} \int _{Q} |f(x)|d\mu(x)\le C\|f\|_{\ast2}
 \end{equation}
 for any cube $ Q\in \mathscr{Q}_{ex} \setminus  \mathscr{Q}  $ by two cases.

Case 1.  $ 1<  \eta <2$. For $\mu$-almost any $x\in Q$, by Proposition \ref{p1.3}(ii), there exists some $(\alpha ,\beta  ) $-doubling cube $Q_{x}$ centered at $x$ with side length $ \alpha ^{-k} (\eta-1)  $, where $k$ is chosen to be the smallest non-negative integer, i.e., $Q_{x}$ is the biggest $(\alpha ,\beta ) $-doubling cube with $l(Q_{x})\le \eta-1$.  From the choice of $Q_{x}$, it's not hard to see that $ l(Q_{x})\le \eta -1 < l(\widetilde{\alpha Q_{x}} )$, then  $l(\alpha^{k}Q_{x}) = \eta -1< l(\alpha^{k+1}Q_{x}) $.

Let $n:=\log_{\alpha }{\frac{1}{\eta-1} } $. By $ 1<  \eta <2$, we have $n>0$. We  find that
\begin{equation*}
\begin{aligned}
&\alpha^{n} l(\alpha^{k}Q_{x}) \le \alpha^{n} (\eta -1) < \alpha^{n} l(\alpha^{k+1}Q_{x}) \Longleftrightarrow   l( \alpha^{n+k}Q_{x}) \le 1 <  l( \alpha^{n+k+1}Q_{x}).\\
\end{aligned}
\end{equation*}
Then the length of the cube $\alpha^{\left \lceil n \right \rceil +k}Q_{x} $ or $\alpha^{\left \lceil n \right \rceil +k+1}Q_{x}$ is in the range of $(1,\alpha]$ and we denote it by $Q_{vx}$. And by the definition of $\mathscr{Q}$ and $\mathscr{Q}_{ex} $, we have
\begin{equation}\label{eq2x.2}
Q_{x}\in \mathscr{Q} , Q_{vx} \in \mathscr{Q}_{ex} \setminus  \mathscr{Q} ,  Q_{x}\subset \alpha^{k}Q_{x} \subset Q_{vx}.
\end{equation}

Notice that $l(Q_{x})\le \eta-1 \le (\eta-1)l(Q)$, then by Lemma \ref{l2.6}, we obtain
\begin{equation}\label{ex3.11}
Q_{x}\subset \eta Q.
\end{equation}

By the Besicovich covering lemma (Lemma \ref{ll2.1}), we can find a family of points $ \left\{ x_{i} \right\}_{i} \subset Q \cap \supp (\mu) $ such that $Q \cap \supp (\mu)$ is covered by a family of cubes $\left \{Q_{x_{i}} \right  \}_{i} $ with bounded overlap. From this, (\ref{eq2.5}), (\ref{eq2.6}), (\ref{eq2x.2}) and Proposition \ref{p1.5}(i), it follows that
\begin{align*}
|f_{Q_{x_{i}}}| \le & |f_{Q_{x_{i}}}-f_{Q_{vx_{i}}}|+|f_{Q_{vx_{i}}}|  \le  K_{Q_{x_{i}},Q_{vx_{i}}} \|f\|_{\ast2}+C\|f\|_{\ast2}\\
\le &  K_{Q_{x_{i}}, \alpha^{k} Q_{x_{i}}}\|f\|_{\ast2}+ K_{\alpha^{k} Q_{x_{i}},Q_{vx_{i}}} \|f\|_{\ast2} +C\|f\|_{\ast2}  .
\end{align*}

The first item on the right hand side is bounded by $ C \|f\|_{\ast2}$ due to the fact that there are no doubling cubes of the form $ \alpha^{m} Q_{x_{i}}$  between $ Q_{x_{i}}  $ and $ \alpha^{k} Q_{x_{i}} $, and hence $ K_{Q_{x_{i}}, \alpha^{k} Q_{x_{i}}} \le C $. The second item is also bounded by $ C \|f\|_{\ast2}$,  because $ \alpha^{k} Q_{x_{i}} $ and $  Q_{vx_{i}} $ have comparable sizes. So,
\begin{equation}\label{eq2.9}
\begin{aligned}
 |f_{Q_{x_{i}}}| \le C \|f\|_{\ast2} .
\end{aligned}
\end{equation}
Since $Q_{x_{i}}$ is $(\alpha ,\beta  ) $-doubling, $ \eta \le \alpha$ and (\ref{ex3.11}), we also have
\begin{equation}\label{eq2.10}
\sum_{i } \mu(\eta Q_{x_{i}})\le \sum_{i } \mu(\alpha Q_{x_{i}}) \le \beta \sum_{i } \mu( Q_{x_{i}}) \le C \mu (\eta Q).
\end{equation}
From (\ref{eq2.9}) and (\ref{eq2.10}), we deduce that
\begin{equation}\label{eq2x.5}
\begin{aligned}
\frac{1}{\mu(\eta Q)} \int _{Q} |f(x)|d\mu(x)  \le & \frac{1}{\mu(\eta Q)} \sum_{i} \int _{Q_{x_{i}}} |f(x)|d\mu(x)\\
\le & \frac{1}{\mu(\eta Q)} \sum_{i} \left [ \int _{Q_{x_{i}}} |f(x)-f_{Q_{x_{i}}}|d\mu(x)+\mu(Q_{x_{i}}) |f_{Q_{x_{i}}}| \right ]  \\
\le &  \frac{1}{\mu(\eta Q)}  \left [\sum_{i} \mu(\eta Q_{x_{i}})\|f\|_{\ast2}+\sum_{i} \mu(Q_{x_{i}}) \|f\|_{\ast2} \right ] \\
\le &  \frac{1}{\mu(\eta Q)} \left [C \mu(\eta Q)\|f\|_{\ast2}+C \mu(\eta Q) \|f\|_{\ast2} \right ] \le C \|f\|_{\ast2},
\end{aligned}
\end{equation}
which implies (\ref{eq2.8}) holds.

Case 2. $\eta \ge 2$. For $\mu$-almost any $x\in Q \in \mathscr{Q}_{ex} \setminus  \mathscr{Q} $, by Proposition \ref{p1.3}(ii), there exists some $(\alpha ,\beta) $-doubling cube $Q_{x}$ centered at $x$ with side length $ \alpha ^{-k} $, where $k $ is chosen to be the smallest non-negative integer. From the choice of $Q_{x}$, it's not hard to see that $ l(Q_{x})\le 1 < l(\widetilde{\alpha Q_{x}} )$, and $(\widetilde{\alpha Q_{x}})'=Q_{x} $. Then,
\begin{equation}\label{eq2x.6}
Q_{x}\in \mathscr{Q}, \ \widetilde{\alpha Q_{x}}\in \mathscr{Q}_{ex} \setminus  \mathscr{Q}, \ Q_{x}\subset \widetilde{\alpha Q_{x}} \ \mathrm{and} \ Q_{x} \subset  2Q \subset \eta Q.
\end{equation}
 By the Besicovich covering lemma (Lemma \ref{ll2.1}), we can find a family of points $ \left\{ x_{i} \right\}_{i} \subset Q \cap \supp (\mu) $ such that $Q \cap \supp (\mu)$ is covered by a family of cubes $\left \{Q_{x_{i}} \right  \}_{i} $ with bounded overlap. From this, (\ref{eq2x.6}) and  Proposition \ref{p1.5}, it follows that
\begin{align*}
|f_{Q_{x_{i}}}| \le & |f_{Q_{x_{i}}}-f_{\widetilde{\alpha Q_{x_{i}}}}|+|f_{\widetilde{\alpha Q_{x_{i}}}}|
 \le  K_{Q_{x_{i}},\widetilde{\alpha Q_{x_{i}}}} \|f\|_{\ast2}+C\|f\|_{\ast2}  \\
 \le & C (K_{Q_{x_{i}},\alpha Q_{x_{i}}}+K_{\alpha Q_{x_{i}},\widetilde{\alpha Q_{x_{i}}}}) \|f\|_{\ast2}+C\|f\|_{\ast2} \le C\|f\|_{\ast2}  .
\end{align*}
By this and repeating the estimate of (\ref{eq2x.5}), we obtain
$$  \frac{1}{\mu(\eta Q)} \int _{Q} |f(x)|d\mu(x) \le C\|f\|_{\ast2}.  $$

Step II. Prove
$$   \frac{1}{\mu(\eta Q)}\int _{Q} |f(x)-m_{\widetilde{Q} }f|d\mu(x)  \le C\|f\|_{\ast2} \qquad  \mathrm{ for \ any \ cube \ } Q\in \mathscr{Q}.$$

If $Q$ is a $(\alpha ,\beta  ) $-doubling cube with $ Q \in \mathscr{Q} $, by (\ref{eq2.4}) and $ \eta \le \alpha$, we have
\begin{equation}\label{eq2.11}
|f_{Q}-m_{Q}f|\le \frac {1}{\mu(Q)} \int _{Q}|f(x)-f_{Q}|d\mu(x) \le \|f\|_{\ast2} \frac {\mu(\eta Q)}{\mu(Q)}\le C \|f\|_{\ast2}.
\end{equation}
Thus for any cube  $Q \in \mathscr{Q}$, if $\widetilde{Q} \in \mathscr{Q} $, by Proposition \ref{p1.5}(iii), (\ref{eq2.5}) and (\ref{eq2.11}), we obtain
\begin{align*}
|f_{ Q}-m_{\widetilde Q}f| \le & |f_{Q}-f_{\widetilde Q}|+|f_{\widetilde Q}-m_{\widetilde Q}f| \\
\le & K_{Q,\widetilde Q}\|f\|_{\ast2}+C\|f\|_{\ast2} \le C\|f\|_{\ast2},
\end{align*}
or if $\widetilde{Q} \in \mathscr{Q}_{ex} \setminus  \mathscr{Q}$, by  (\ref{eq2.5}), (\ref{eq2.6}), (\ref{eq2.8}) and $\eta\le \alpha$, we have \begin{align*}
|f_{ Q}-m_{\widetilde Q}f| \le & |f_{Q}-f_{\widetilde Q}|+|f_{\widetilde Q}-m_{\widetilde Q}f| \\
\le & K_{Q,\widetilde Q}\|f\|_{\ast2}+|f_{\widetilde Q}|+|m_{\widetilde Q}f|\\
\le &  C\|f\|_{\ast2}+C\|f\|_{\ast2}+\frac{\mu(\eta \widetilde Q)}{\mu(\widetilde Q)}\|f\|_{\ast2}\le C\|f\|_{\ast2}.
\end{align*}

Therefore, if $Q\in \mathscr{Q}$, we obtain
\begin{align*}
\frac{1}{\mu(\eta Q)}\int _{Q}|f(x)-m_{\widetilde Q}f|d\mu(x) \le & \frac{1}{\mu(\eta Q)}\int _{Q}(|f(x)-f_{Q}|+|f_{Q}-m_{\widetilde Q}f|)d\mu(x)\\
 \le & C\|f\|_{\ast2}+C \frac {\mu ( Q)}{\mu(\eta Q)} \|f\|_{\ast2}  \le C\|f\|_{\ast2}  .
\end{align*}

Step III.  Prove  $$ | m_{Q}f-m_{R}f   |  \le CK_{Q,R} \|f\|_{\ast2} $$
for any two $(\alpha ,\beta ) $-doubling cubes $Q\subset R $ with $Q\in \mathscr{Q}$ and $R\in \mathscr{Q}_{ex} $.

    If $R \in \mathscr{Q}$, by (\ref{eq2.11}) and (\ref{eq2.5}), we have
    \begin{align*}
|m_{Q}f-m_{R}f| \le & |m_{Q}f-f_{Q}|+|f_{Q}-f_{R}|+|f_{R}-m_{R}f| \\
 \le & C \|f\|_{\ast2}+CK_{Q,R}\|f\|_{\ast2}+C\|f\|_{\ast2}  \le CK_{Q,R}\|f\|_{\ast2}  .
\end{align*}
 If $ R\in \mathscr{Q}_{ex} \setminus  \mathscr{Q}$, by (\ref{eq2.11}), (\ref{eq2.5}), (\ref{eq2.6}), (\ref{eq2.8}) and $\eta \le \alpha$, we obtain
 \begin{align*}
|m_{Q}f-m_{R}f| \le & |m_{Q}f-f_{Q}|+|f_{Q}-f_{R}|+|f_{R}-m_{R}f| \\
 \le & C \|f\|_{\ast2}+CK_{Q,R}\|f\|_{\ast2}+|f_{R}|+|m_{R}f| \\
 \le & C\|f\|_{\ast2}+CK_{Q,R}\|f\|_{\ast2}+\|f\|_{\ast2}+ C\frac{\mu(\eta R)}{\mu(R)} \|f\|_{\ast2}  \le CK_{Q,R}\|f\|_{\ast2}  .
\end{align*}

Combining Steps I$-$III, we finally obtain that $  \|f\|_{\ast1}  \le C\|f\|_{\ast2}$.

Therefore we finish the proof of Theorem \ref{l2.10}.
\end{proof}

In the following definition, we will further introduce two $rbmo^{1}_{\mathscr Q}(\mu)$-type spaces  related with general cubes and ($\alpha,\beta$)-doubling cubes, respectively.

\begin{definition}\label{d2.11}
Let $ \eta \in (1,\infty)$ and $ \alpha \in [\eta,\infty)$. $ \mathscr{Q}:=\left \{Q:l(Q)\le 1\right \}$ and $ \mathscr{Q}_{ex}:=\left \{Q:l(Q')\le 1\right \}$.
\begin{itemize}
\item[\rm(i)] We say that $f\in L_{loc}(\mu)$ is in $rbmo^{3}_{\mathscr Q}(\mu) $ if there exists some constant $C_{5}\ge 0 $ such that, for any cube   $Q\in \mathscr{Q} $,
\begin{equation}\label{eq2.12}
\frac{1}{\mu(\eta Q)}\int _{Q}\left | f(x)-m_{Q}f  \right | d\mu(x) \le C_{5},
\end{equation}
for any two cubes $Q \subset R$ with $Q\in \mathscr{Q}$ and $R\in \mathscr{Q}_{ex}$,
\begin{equation}\label{eq2.13}
|m_{Q}f-m_{R}f| \le C_{5} K_{Q,R} \left[\frac{\mu(\eta Q)}{\mu(Q)}+\frac{\mu(\eta R)}{\mu(R)}\right],
\end{equation}
and for any cube   $Q\in \mathscr{Q}_{ex} \setminus  \mathscr{Q}$,
\begin{equation}\label{eq2.14}
\frac{1}{\mu(\eta Q)} \int _{Q}\left | f(x) \right | d\mu(x) \le C_{5} .
\end{equation}
Moreover, we define the $rbmo^{3}_{\mathscr Q}(\mu)$ norm of $f$ by the minimal constant $C_{5}$, and write it by $\|f\|_{\ast3}$.
\item[\rm(ii)] We say that $f\in L_{loc}(\mu)$ is in $rbmo^{4}_{\mathscr Q}(\mu) $ if there exists some constant $C_{6}\ge 0 $ such that, for any $(\alpha ,\beta) $-doubling cube  $Q\in \mathscr{Q} $,
\begin{equation}\label{eq2.15}
\frac{1}{\mu( Q)}\int _{Q}\left | f(x)-m_{Q}f  \right | d\mu(x) \le C_{6},
\end{equation}
for any two $(\alpha ,\beta) $-doubling cubes $Q \subset R$ with $Q\in \mathscr{Q}$ and $R\in \mathscr{Q}_{ex}$,
\begin{equation}\label{eq2.16}
|m_{Q}f-m_{R}f| \le C_{6} K_{Q,R} ,
\end{equation}
and for any $(\alpha ,\beta) $-doubling cube  $Q\in \mathscr{Q}_{ex} \setminus  \mathscr{Q}$,
\begin{equation}\label{eq2.17}
\frac{1}{\mu( Q)}\int _{Q}\left | f(x) \right | d\mu(x) \le C_{6} .
\end{equation}
Moreover, we define the $rbmo^{4}_{\mathscr Q}(\mu)$ norm of $f$ by the minimal constant $C_{6}$, and write it by $\|f\|_{\ast4}$.
\end{itemize}
\end{definition}
\begin{theorem}\label{l2.12}
Let $\eta \in (1,\infty) $ and $\alpha \in[\eta,\infty)$. Then $ rbmo^{1}_{\mathscr Q}(\mu)=rbmo^{3}_{\mathscr Q}(\mu)=rbmo^{4}_{\mathscr Q}(\mu)$ with equivalent norms.
\end{theorem}
\begin{proof}
Let us first prove that $rbmo^{1}_{\mathscr Q}(\mu)\subset rbmo^{3}_{\mathscr Q}(\mu)$, i.e., $\|f\|_{\ast3}\le \|f\|_{\ast1}$ for any $f \in rbmo^{1} _{\mathscr Q}(\mu)$ by three steps.

Step 1.  Prove
\begin{equation}\label{eq2x.3}
\frac{1}{\mu(\eta Q)}\int _{Q}\left | f(x)-m_{Q}f  \right | d\mu(x) \le C\|f\|_{\ast1}
\end{equation}
for any cube $Q\in \mathscr{Q} $.

By (\ref{eq2.1}), we obtain
 \begin{align*}
 \frac{1}{\mu(\eta Q)}\int_{Q}|f(x)-m_{Q}f|d\mu(x) \le & \frac{1}{\mu(\eta Q)}\int_{Q}|f(x)-m_{\widetilde Q}f|d\mu(x)+\frac{\mu(Q)}{\mu(\eta Q)}|m_{Q}f-m_{\widetilde Q}f|\\
 \le & \|f\|_{\ast1}+\frac{\mu(Q)}{\mu(\eta Q)} \frac{1}{\mu(Q)}\int_{Q}|f(x)-m_{\widetilde Q}f|d\mu(x)  \le  2\|f\|_{\ast1},
 \end{align*}
  then (\ref{eq2x.3}) holds with $C=2$.

  Step 2. Prove $$|m_{Q}f-m_{R}f| \le C K_{Q,R} \left[\frac{\mu(\eta Q)}{\mu(Q)}+\frac{\mu(\eta R)}{\mu(R)}\right]\|f\|_{\ast1}$$
  for any two cubes $Q \subset R$ with $Q\in \mathscr{Q}$ and $R\in \mathscr{Q}_{ex}$.

  It is easy to deduce that
$$    |m_{Q}f-m_{R}f| \le |m_{Q}f-m_{\widetilde Q}f|+|m_{\widetilde Q}f-m_{\widetilde R}f|+|m_{R}f-m_{\widetilde R}f|=:\mathrm{I+II+III} .  $$

By (\ref{eq2.1}), we have
\begin{equation}\label{eq2.18}
\mathrm{I}=|m_{Q}f-m_{\widetilde Q}f| \le \frac {1}{\mu(Q)}\int_{Q}|f(x)-m_{\widetilde Q}f|d\mu(x) \le \frac{\mu(\eta Q)}{\mu(Q)}\|f\|_{\ast1 } .
\end{equation}

By Proposition \ref{p1.5}(iii), we obtain $K_{Q',\widetilde Q}\le C$, $K_{Q',Q}\le C$ and $K_{R,\widetilde R}\le C$. From this and (\ref{eq2.2}), it follows that
 \begin{align*}
\mathrm{II}= &|m_{\widetilde Q}f-m_{\widetilde R}f| \le  |m_{Q'}f-m_{\widetilde Q}f|+|m_{Q'}f-m_{\widetilde R}f| \\
 \le & K_{Q',\widetilde Q}\|f\|_{\ast1}+K_{Q',\widetilde R}\|f\|_{\ast1}
 \le  C \|f\|_{\ast1} +C(K_{Q',Q}+K_{Q,R }+K_{R,\widetilde R})\|f\|_{\ast1}   \\
 \le & C\|f\|_{\ast1} +CK_{Q,R}\|f\|_{\ast1} \le CK_{Q,R}\|f\|_{\ast1}.
\end{align*}

Now let us estimate $\mathrm{III}$ by two cases.

When $ R \in \mathscr{Q}$, by (\ref{eq2.18}), we have
$$\mathrm{III}=|m_{R}f-m_{\widetilde R}f|  \le \frac{\mu(\eta R)}{\mu(R)}\|f\|_{\ast1 } .$$

When $ R \in \mathscr{Q}_{ex} \setminus  \mathscr{Q}$, by (\ref{eq2.3}) and $\eta\le \alpha$, we obtain
 \begin{align*}
\mathrm{III}=|m_{R}f-m_{\widetilde R}f| \le & |m_{R}f| + |m_{\widetilde R}f| \le \frac {\mu(\eta R)}{\mu(R)} \|f\|_{\ast1}+\frac {\mu(\eta \widetilde R)}{\mu(\widetilde R)} \|f\|_{\ast1}\\
 \le & \frac {\mu(\eta R)}{\mu(R)} \|f\|_{\ast1}+C\|f\|_{\ast1} \le \frac {\mu(\eta R)}{\mu(R)} \|f\|_{\ast1}+CK_{Q,R}\|f\|_{\ast1}.
\end{align*}

Combining $\mathrm{I}$-$\mathrm{III}$, we conclude that
 \begin{align*}
|m_{Q}f-m_{ R}f| \le & |m_{Q}f-m_{\widetilde Q}f|+|m_{\widetilde Q}f-m_{\widetilde R}f|+|m_{R}f-m_{\widetilde R}f| \\
\le &\frac {\mu(\eta Q)}{\mu(Q)} \|f\|_{\ast1}+CK_{Q,R}\|f\|_{\ast1} +\frac {\mu(\eta R)}{\mu(R)} \|f\|_{\ast1}+CK_{Q,R}\|f\|_{\ast}\\
 \le &   CK_{Q,R}\left   [ \frac {\mu(\eta Q)}{\mu(Q)}+\frac {\mu(\eta R)}{\mu(R)} \right  ] \|f\|_{\ast1}.
\end{align*}

Step 3. Prove $\frac{1}{\mu(\eta Q)}\int _{Q}\left | f(x) \right | d\mu(x) \le C\|f\|_{\ast1}$ for any cube   $Q\in \mathscr{Q}_{ex} \setminus  \mathscr{Q}$.

In fact, it is obviously true by (\ref{eq2.3}).

Combining Steps 1-3, we  obtain that $  \|f\|_{\ast3}  \le C\|f\|_{\ast1}$ and hence  $rbmo^{1}_{\mathscr Q}(\mu)\subset rbmo^{3}_{\mathscr Q}(\mu)$.

The implication $rbmo^{3}_{\mathscr Q}(\mu)\subset rbmo^{4}_{\mathscr Q}(\mu)$ is obvious: One only has to consider $(\alpha ,\beta) $-doubling cubes in $rbmo^{3}_{\mathscr Q}(\mu)$.

Let us now prove that $rbmo^{4}_{\mathscr Q}(\mu)\subset rbmo^{1}_{\mathscr Q}(\mu)$, i.e., $\|f\|_{\ast1}\le C\|f\|_{\ast4}$ for any $f\in rbmo^{4}_{\mathscr Q}$ by three steps.

Step A.
Prove
\begin{equation}\label{eq2x.4}
\frac{1}{\mu(\eta Q)}\int_{Q}|f(x)|d\mu(x)\le C\|f\|_{\ast4}
\end{equation}
for any cube   $Q\in \mathscr{Q}_{ex} \setminus  \mathscr{Q}$.

Let $ Q $ be a cube with $ Q \in \mathscr{Q}_{ex} \setminus  \mathscr{Q} $. We consider the same covering as in the proof of (\ref{eq2.8}). Since $Q_{vx_{i}}, \widetilde Q_{vx_{i}} \in \mathscr{Q}_{ex} \setminus  \mathscr{Q}$, by (\ref{eq2.16}), (\ref{eq2.17}) and (i)-(iii) of Proposition \ref{p1.5}, we  have
\begin{equation}\label{eq2.19}
 \begin{aligned}
|m_{Q_{x_{i}}}f| \le & |m_{Q_{x_{i}}}f-m_{\widetilde Q_{vx_{i}}}f|+|m_{\widetilde Q_{vx_{i}}}f| \\
\le & K_{Q_{x_{i}},\widetilde Q_{vx_{i}}}\|f\|_{\ast4}+C\|f\|_{\ast4} \le  C K_{Q_{x_{i}},\widetilde Q_{vx_{i}}}\|f\|_{\ast4}\\
\le & C(K_{Q_{x_{i}}, Q_{vx_{i}}}+K_{Q_{vx_{i}},\widetilde Q_{vx_{i}}})\|f\|_{\ast4}\\
\le & C(K_{Q_{x_{i}}, \alpha^{k}Q_{x_{i}}}+K_{\alpha^{k}Q_{x_{i}},Q_{vx_{i}}}+K_{Q_{vx_{i}},\widetilde Q_{vx_{i}}})\|f\|_{\ast4}\le C\|f\|_{\ast4}.
\end{aligned}
\end{equation}

From this,   (\ref{eq2.19}), (\ref{eq2.10}) and (\ref{eq2.15}), it follows that
\begin{align*}
\frac{1}{\mu(\eta Q)} \int _{Q} |f(x)|d\mu(x)  \le & \frac{1}{\mu(\eta Q)} \sum_{i} \int _{Q_{x_{i}}} |f(x)|d\mu(x)\\
\le & \frac{1}{\mu(\eta Q)} \sum_{i} \left [ \int _{Q_{x_{i}}} |f(x)-m_{Q_{x_{i}}}f|d\mu(x)+\mu(Q_{x_{i}}) |m_{Q_{x_{i}}}f| \right ]  \\
\le &  \frac{1}{\mu(\eta Q)}  \left [\sum_{i} \mu(\eta Q_{x_{i}})\|f\|_{\ast4}+\sum_{i} C \mu(Q_{x_{i}}) \|f\|_{\ast4} \right ] \\
\le &  \frac{1}{\mu(\eta Q)} \left [C \mu(\eta Q)\|f\|_{\ast4}+C \mu(\eta Q) \|f\|_{\ast4} \right ] \le C \|f\|_{\ast4},
\end{align*}
which implies (\ref{eq2x.4}) holds true.

Step B. Prove $|m_{Q}f-m_{R}f| \le C K_{Q,R}\|f\|_{\ast4} $ for any two $(\alpha ,\beta) $-doubling cubes $Q \subset R$ with $Q\in \mathscr{Q}$ and $R\in \mathscr{Q}_{ex}$.

It is obviously true since (\ref{eq2.2}) and (\ref{eq2.16}) are same with each other.

Step C.
Prove
\begin{equation}\label{eq2x.9}
\frac{1}{\mu(\eta Q)}\int_{Q}|f(x)-m_{\widetilde Q}f|d\mu(x)\le C\|f\|_{\ast4} \qquad    \mathrm{ for \ any \ cube} \  Q\in \mathscr{Q} .
\end{equation}

Let $Q$ be a cube in $\mathscr Q$  which is not $(\alpha,\beta)$-doubling or else (\ref{eq2x.9}) holds automatically.
For $\mu$-almost any $x\in Q$, by Proposition \ref{p1.3}(ii), there exists some $(\alpha ,\beta  ) $-doubling cube $Q_{x}$ centered at $x$ with side length $ \alpha ^{-k} (\eta-1) l(Q) $, where $k$ is chosen to be the smallest non-negative integer. By the definition of $Q_{x}$, it is easy to find that $\alpha Q_{x},\cdots, \alpha^{k}Q_{x}$ are not $(\alpha,\beta)$-doubling cubes.    By Lemma \ref{l2.6} with $l(\alpha^{k}Q_{x})=(\eta-1)l(Q)$ and $\eta \le \alpha$,  we  obtain
\begin{equation}\label{exx3.1}
 Q_{x} \subset \alpha^{k}Q_{x}\subset \eta Q \subset \alpha Q \subset \widetilde Q.
\end{equation}
By this and Proposition \ref{p1.5},  we have
\begin{equation}\label{eq2.20}
K_{Q_{x},\widetilde Q} \le C(K_{Q_{x},\alpha^{k}Q_{x}}+K_{\alpha^{k}Q_{x},\eta Q}+K_{\eta Q,\alpha Q}+K_{\alpha Q,\widetilde Q})\le C .
\end{equation}
The first and forth items on the right hand side are bounded by $C$ because there are no ($\alpha,\beta$)-doubling cubes of the form $\alpha^{n}Q_{x} $ between $Q_{x}$ and $\alpha^{k}Q_{x}$ and no $(\alpha,\beta)$-doubling cubes of the form $\alpha^{m}Q$ between $\alpha Q$ and $\widetilde Q$. The second and third items are also bounded by $C$ due to the fact that $\alpha^{k}Q_{x} $ and $\eta Q$ on the one hand and $\eta Q$ and $\alpha Q$ on the other hand have comparable sizes.

Then, by (\ref{eq2.20}), we have
\begin{equation}\label{eq2.21}
|m_{Q_{x}}f-m_{\widetilde Q}f| \le  K_{Q_{x},\widetilde Q}\|f\|_{\ast4}  \le C\|f\|_{\ast4} .
\end{equation}
By the Besicovich covering lemma (Lemma \ref{ll2.1}), there exists a family of  points $\left\{x_{i}\right\} \subset Q$ such that $Q \cap \supp (\mu)$ is covered by a family of cubes $\left \{Q_{x_{i}} \right  \}_{i} $ with bounded overlap. By (\ref{eq2.16}), (\ref{eq2.17}), (\ref{eq2.21}) and (\ref{exx3.1}), we obtain
\begin{align*}
\int_{Q}|f(x)-m_{\widetilde Q}f|d\mu(x) \le & \sum_{i} \int_{Q_{x_{i}}}|f(x)-m_{\widetilde Q}f|d\mu(x)\\
\le &  \sum_{i} \int_{Q_{x_{i}}}|f(x)-m_{Q_{x_{i}}}f|d\mu(x) +\sum_{i}|m_{\widetilde Q}f(x)-m_{Q_{x_{i}}}f|\mu(Q_{x_{i}})  \\
\le &  \sum_{i} \mu(Q_{x_{i}}) C\|f\|_{\ast4}+  \sum_{i}C\|f\|_{\ast4}\mu(Q_{x_{i}})   
\le   C\|f\|_{\ast4}\mu(\eta Q).
\end{align*}

Combining Steps A-C, then we obtain that $\|f\|_{\ast1}\le C\|f\|_{\ast4}$ and hence $rbmo^{4}_{\mathscr Q}(\mu)\subset rbmo^{1}_{\mathscr Q}(\mu)$.

Finally, by $rbmo^{1}_{\mathscr Q}(\mu)\subset rbmo^{3}_{\mathscr Q}(\mu)$, $rbmo^{3}_{\mathscr Q}(\mu)\subset rbmo^{4}_{\mathscr Q}(\mu)$ and $rbmo^{4}_{\mathscr Q}(\mu)\subset rbmo^{1}_{\mathscr Q}(\mu)$, we  obtain that $ rbmo^{1}_{\mathscr Q}(\mu)=rbmo^{3}_{\mathscr Q}(\mu)=rbmo^{4}_{\mathscr Q}(\mu)$ with equivalent norms.

\end{proof}

\section{$rbmo^{1}_{\mathscr Q}(\mu)=rbmo(\mu) \subset RBMO(\mu)$ \label{s3}}
The main aim of this section is to obtain the containment relationships of $RBMO(\mu)$, $rbmo(\mu)$ and $rbmo^{1}_{\mathscr Q}(\mu)$.

\begin{theorem}\label{t3.1}
$rbmo^{1}_{\mathscr Q}(\mu)=rbmo(\mu) $ with equivalent norms.
\end{theorem}
To prove  Theorem \ref{t3.1}, we will use the following lemma which is easy to check by Definition 1.3.
\begin{lemma}\label{l3.2}
For any cube $Q$ in $\mathbb R^{d}$, if $R$ and $S$ are two cubes in $\mathbb R^{d}$ with $l(R)=l(S)$, $Q\subset R$ and $Q\subset S$, then $K_{Q,R}=K_{Q,S}=K_{Q,2^{N_{Q,R}}Q}=K_{Q,2^{N_{Q,S}}Q}$.
\end{lemma}
\begin{proof}[Proof of Theorem \ref{t3.1}]
By the definitions of $rbmo^{1}_{\mathscr Q}(\mu)$ and $rbmo(\mu) $, it is easy to find that $rbmo(\mu) \subset rbmo^{1}_{\mathscr Q}(\mu) $, i.e.,  $\|f\|_{\ast1}\le C \|f\|_{rbmo(\mu)}$ for any $f \in rbmo(\mu)$, so we only need to prove
\begin{equation}
rbmo^{1}_{\mathscr Q}(\mu) \subset rbmo(\mu),\ \mathrm{ i.e.}, \ \|f\|_{rbmo(\mu)} \le C\|f\|_{\ast1} \mathrm {\ for \ any}\ f \in rbmo_{\mathscr Q}(\mu).
\end{equation}

We first prove that if $Q$ is a cube with $l(Q) >1$, then
\begin{equation}\label{eq3.1}
 \frac{1}{\mu(\eta Q)} \int _{Q}\left | f(x)   \right | d\mu(x) \le C\|f\|_{\ast1}.
\end{equation}
In fact, if $Q \in \mathscr Q_{ex} \setminus  \mathscr Q$, (3.1) is obviously true. If $Q \notin \mathscr Q_{ex}  $, we prove (\ref{eq3.1}) under the following two cases.

Case(i) \quad  $\frac {\eta-1}{\eta}l(Q)>1$. For any $x\in Q \cap \supp (\mu)$, let $Q_{x}$ be a cube centered at $x$ with $l(Q_{x})=\alpha^{-N_{1}} \frac {\eta-1}{\eta}l(Q)$ for some integer $N_{1} \ge 0 $ such that $1<l(Q_{x})\le \alpha$. Then $l(\eta Q_{x})\le l((\eta-1)Q)$, and hence $\eta Q_{x} \subset \eta Q$ by Lemma \ref{l2.6}.
    By the Besicovich covering lemma (Lemma \ref{ll2.1}), we can find a family of points $ \left\{ x_{i} \right\}_{i} \subset Q \cap \supp (\mu) $ such that $Q \cap \supp (\mu)$ is covered by a family of cubes $\left \{Q_{x_{i}} \right  \}_{i} $ with bounded overlap. Then by Proposition \ref{p2.4}(iv) and Lemma \ref{l2.6}, we  obtain $Q_{x_{i}}\in \mathscr Q_{ex} \setminus   \mathscr Q$ and $\eta Q_{x_{i}} \subset \eta Q$. By this and (\ref{eq2.3}), we have
     \begin{equation}\label{eq3.9}
 \begin{aligned}
\frac{1}{\mu(\eta Q)} \int _{Q} |f(x)|d\mu(x)  \le & \frac{1}{\mu(\eta Q)} \sum_{i} \int _{Q_{x_{i}}} |f(x)|d\mu(x)  \\
= & \frac{1}{\mu(\eta Q)} \sum_{i}  \mu(\eta Q_{x_{i}})  \frac{1}{\mu(\eta Q_{x_{i}})}  \int _{Q_{x_{i}}} |f(x)|d\mu(x) \\
\le &  \frac{1}{\mu(\eta Q)} \sum_{i}  \mu(\eta Q_{x_{i}})\|f\|_{\ast1} \le C \frac{1}{\mu(\eta Q)} \mu(\eta Q) \|f\|_{\ast1}\\
\le & C\|f\|_{\ast1},
\end{aligned}
\end{equation}
which implies (\ref{eq3.1}) holds true.

Case(ii) \quad  $\frac {\eta-1}{\eta}l(Q)\le 1$. For $\mu$-almost any $x\in Q$, by Proposition \ref{p1.3}(ii), there exists some $(\alpha ,\beta  ) $-doubling cube $Q_{x}$ centered at $x$ with side length $ \alpha ^{-k} \frac {\eta-1}{\eta} $, where $k$ is chosen to be the smallest non-negative integer. From the choice  of $Q_{x}$, it's easy to see that $ l(Q_{x})\le\frac {\eta-1}{\eta} < l(\widetilde{\alpha Q_{x}} )$, then  $l(\alpha^{k}Q_{x}) = \frac {\eta-1}{\eta}< l(\alpha^{k+1}Q_{x}) $.

    Let $n:=\log_{\alpha }{\frac{\eta}{\eta-1} }$. By $\eta>1$, we have $n>0$. We find that
    \begin{equation*}
\begin{aligned}
&\alpha^{n} l(\alpha^{k}Q_{x}) \le \alpha^{n} \frac {\eta-1}{\eta} < \alpha^{n} l(\alpha^{k+1}Q_{x})
\Longleftrightarrow  l( \alpha^{n+k}Q_{x}) \le 1 <  l( \alpha^{n+k+1}Q_{x}).\\
\end{aligned}
\end{equation*}
Then the length of the cube $\alpha^{\left \lceil n \right \rceil +k}Q_{x} $ or $\alpha^{\left \lceil n \right \rceil +k+1}Q_{x}$ is in the range of $(1,\alpha]$ and we denote it by $Q_{vx}$.
 Then we can obtain $Q_{x}\in \mathscr Q$, $Q_{vx}\in \mathscr Q_{ex}\setminus \mathscr Q$, $Q_{x}\subset Q_{vx}$ and $\eta Q_{x} \subset \eta Q$ by $l(Q_{x})\le \frac {\eta-1}{\eta} \le \frac {\eta-1}{\eta}l(Q) $ and Lemma \ref{l2.6}. By the Besicovich covering lemma (Lemma \ref{ll2.1}), we  find a family of points $ \left\{ x_{i} \right\}_{i} \subset Q \cap \supp (\mu) $ such that $Q \cap \supp (\mu)$ is covered by a family of cubes $\left \{Q_{x_{i}} \right  \}_{i} $ with bounded overlap. Moreover, by (\ref{eq2.19}) and Theorem \ref{l2.12},  we have $|m_{Q_{x_{i}}}f|\le C\|f\|_{\ast1}$.
From this,  (\ref{eq2.1}) and $\eta Q_{x_{i}}\subset \eta Q$, we deduce that
\begin{align*}
\frac{1}{\mu(\eta Q)} \int _{Q} |f(x)|d\mu(x)  \le & \frac{1}{\mu(\eta Q)} \sum_{i} \int _{Q_{x_{i}}} |f(x)|d\mu(x)\\
\le & \frac{1}{\mu(\eta Q)} \sum_{i} \left [ \int _{Q_{x_{i}}} |f(x)-m_{Q_{x_{i}}}f|d\mu(x)+\mu(Q_{x_{i}}) |m_{Q_{x_{i}}}f| \right ]  \\
\le &  \frac{1}{\mu(\eta Q)} \sum_{i} \left [ \mu(\eta Q_{x_{i}})\|f\|_{\ast1}+C \mu(Q_{x_{i}}) \|f\|_{\ast1} \right ] \\
\le &  \frac{1}{\mu(\eta Q)} [C{\mu(\eta Q)}\|f\|_{\ast1}+C{\mu(\eta Q)}\|f\|_{\ast1}] \le C \|f\|_{\ast1},
\end{align*}
which implies (\ref{eq3.1}) holds true.

Next,  let us prove
\begin{equation}\label{eq3.3}
|m_{Q}f-m_{R}f| \le K_{Q,R} \|f\|_{\ast1}
\end{equation}
for any two $(\alpha ,\beta) $-doubling cubes $Q\subset R $ with $l(Q)\le1$.

If $R \in \mathscr Q_{ex} \setminus  \mathscr Q$, (\ref{eq3.3}) is obviously true, so we only need to consider $R \notin \mathscr Q_{ex} $.

For any cube $R \notin \mathscr Q_{ex} $ with $Q\subset R$, we can obtain a cube $2^{N_{Q,R}}Q$ where $N_{Q,R} $ is the first integer $k$ such that $l(2^{k}Q)\ge l(R)$, then $K_{Q,R}=K_{Q,2^{N_{Q,R}}Q}$ by Lemma \ref{l3.2}. Moreover, by $l(2^{N_{Q,R}}Q)\ge l(R)>1$, we know that $2^{N_{Q,R}}Q \notin \mathscr Q$.  Here we remark that whether $2^{N_{Q,R}}Q \in \mathscr Q_{ex} \setminus  \mathscr Q$ is unknown, so we will consider  (3.3) in the following two cases.

Case(i) \quad If $2^{N_{Q,R}}Q \in \mathscr Q_{ex} \setminus  \mathscr Q$,  then $\widetilde{2^{N_{Q,R}}Q }\in \mathscr Q_{ex} \setminus  \mathscr Q$ by Proposition 2.4(ii). By (\ref{eq2.2}),  (\ref{eq2.3}),  (\ref{eq3.1}),  $\eta \le \alpha$,  (i) and (iii) of Proposition \ref{p1.5} and Lemma \ref{l3.2}, we obtain
\begin{align*}
|m_{Q}f-m_{R}f| \le & |m_{Q}f|+|m_{R}f|\le |m_{Q}f-m_{\widetilde{2^{N_{Q,R}}Q }}f|+|m_{\widetilde{2^{N_{Q,R}}Q }}f|+|m_{R}f|\\
\le & K_{Q,\widetilde{2^{N_{Q,R}}Q }}\|f\|_{\ast1}+\frac {\mu(\eta \widetilde{2^{N_{Q,R}}Q })}{\mu(\widetilde{2^{N_{Q,R}}Q })}\frac {1}{\mu(\widetilde{2^{N_{Q,R}}Q })} \int _{\widetilde{2^{N_{Q,R}}Q }} |f(x)| d\mu(x) \\
+ & \frac {\mu(\eta R)}{\mu(R)} \frac {1}{\mu(\eta R)}\int _{R}|f(x)| d\mu(x) \\
\le & (K_{Q,{2^{N_{Q,R}}Q }} +  K_{2^{N_{Q,R}}Q,\widetilde {2^{N_{Q,R}}Q }}      )      \|f\|_{\ast1}+C\|f\|_{\ast1}+C\|f\|_{\ast1} \\
\le  & (K_{Q,R}+C)\|f\|_{\ast1}+C\|f\|_{\ast1} \le CK_{Q,R}\|f\|_{\ast1},
\end{align*}
which implies (\ref{eq3.3}) holds true.

Case(ii) \quad  If $2^{N_{Q,R}}Q \notin \mathscr Q_{ex} \setminus  \mathscr Q$, it means that $l((2^{N_{Q,R}}Q)')>1 $. Let $Q_{0}$ be a cube centered at $Q$ with  $l(Q_{x})=\alpha^{-N_{3}} l(2^{N_{Q,R}}Q)$ for some integer $N_{3} \ge 0 $ such that $1<l(Q_{0})\le \alpha$. Then  $\widetilde Q_{0}\in \mathscr Q_{ex} \setminus  \mathscr Q$,  $ l(Q_{0}') \le 1 $   and
\begin{equation}\label{3x.2}
Q \subset Q_{0} \subset \widetilde Q_{0} \subset (2^{N_{Q,R}}Q)' \subset 2^{N_{Q,R}}Q
\end{equation}
by Remark \ref{r2.5}(ii). From this, Proposition \ref{p1.5}(i) and Lemma \ref{l3.2}, it follows that
 \begin{equation}\label{eq3.9}
 \begin{aligned}
K_{Q,\widetilde Q_{0}}\le K_{Q,2^{N_{Q,R}}Q} =K_{Q,R}.
\end{aligned}
\end{equation}
Consequently, by (\ref{3x.2}),  (\ref{eq2.2}), $\eta \le \alpha $, (\ref{eq3.1}) with $l(R)>1$ and (\ref{eq3.9}), we obtain
\begin{align*}
|m_{Q}f-m_{R}f| \le & |m_{Q}f-m_{\widetilde Q_{0}}f|+|m_{\widetilde Q_{0}}f|+|m_{R}f| \\
\le & K_{Q,\widetilde Q_{0}}\|f\|_{\ast1}+\frac{\mu(\eta \widetilde Q_{0})}{\mu(\widetilde Q_{0})} \|f\|_{\ast1}+\frac{\mu(\eta R)}{\mu(R)}\|f\|_{\ast1}  \\
\le & K_{Q,R}\|f\|_{\ast1} +C\|f\|_{\ast1}+C\|f\|_{\ast1} \le C K_{Q,R} \|f\|_{\ast1},
\end{align*}
which implies (\ref{eq3.3}) holds true.

Finally, (2.1) and (2.4) are same with each other. From this, (\ref{eq3.1}) and (\ref{eq3.3}), it follows that $rbmo^{1}_{\mathscr Q}(\mu)=rbmo(\mu) $ with equivalent norms.
\end{proof}

Yang \cite{5} pointed out that it is easy to see $rbmo(\mu)\subset RBMO(\mu)$. However, the proof seems not trivial from their definitions. We add the proof for the sake of completeness.

\begin{theorem}\label{t3.x1}
Let $RBMO(\mu)$ and $rbmo(\mu)$ be, respectively, defined as in Definitions \ref{d1.1} and \ref{d1.6}  with $1< \eta \le \alpha < \infty$. Then $rbmo(\mu)\subset RBMO(\mu)$.
\end{theorem}
\begin{proof}
 Let us prove that, for any $f\in rbmo(\mu)$,
$$\frac{1}{\mu(\eta Q)}\int_{Q}|f(x)-m_{\widetilde Q}f|d\mu(x)\le C \|f\|_{rbmo(\mu)} \qquad \mathrm{for \ any \ cube } \ Q\subset \mathbb R^{d} .$$

If $l(Q)\le 1$, it is obviously true.

If $l(Q)>1$, it is easy to find that $l(\widetilde Q)>1$. So by (\ref{eq1.6}) and $\eta \le \alpha$, we  obtain
\begin{equation}\label{eq3x.1}
\begin{aligned}
\frac{1}{\mu(\eta Q)}\int_{Q}|f(x)-m_{\widetilde Q}f|d\mu(x) \le & \frac{1}{\mu(\eta Q)}\int_{Q}(|f(x)|+|m_{\widetilde Q}f|)d\mu(x) \\
\le & \frac{1}{\mu(\eta Q)}\int_{Q}|f(x)|d\mu(x)+\frac{\mu(Q)}{\mu(\eta Q)}|m_{\widetilde Q}f| \\
\le & C\|f\|_{rbmo(\mu)}.
\end{aligned}
\end{equation}

Then let us prove that, for any two $(\alpha,\beta)$-doubling cubes $Q\subset R$,
$$|m_{Q}f-m_{R}f|\le CK_{Q,R}\|f\|_{rbmo(\mu)}.$$

If $l(Q)\le 1$, it is obviously true.

If $l(Q)>1$, by (\ref{eq1.6}), $K_{Q,R}\ge 1$ and $\eta \le \alpha$, we  obtain
\begin{align*}
|m_{Q}f-m_{R}f|\le & |m_{Q}f|+|m_{R}f|\le m_{Q}|f|+m_{R}|f|  \\
\le & \frac{\mu(\eta Q)}{\mu(Q)}\frac{1}{\mu(\eta Q)}\int_{Q}|f(x)|d\mu(x) +\frac{\mu(\eta R)}{\mu(R)}\frac{1}{\mu(\eta R)}\int_{R}|f(x)|d\mu(x) \\
\le & C\|f\|_{rbmo(\mu)} \le CK_{Q,R}\|f\|_{rbmo(\mu)}.
\end{align*}
From this and (\ref{eq3x.1}), it follows that  $\|f\|_{RBMO(\mu)}\le C\|f\|_{rbmo(\mu)} $ for any $f\in rbmo(\mu)$ and hence $rbmo(\mu) \subset RBMO(\mu) $.
\end{proof}

By Theorems \ref{l2.10},  \ref{l2.12} and \ref{t3.1}, we know that the four bmo-type spaces, $rbmo^{i}_{\mathscr Q}(\mu)$ $(i= \mathrm 1,\mathrm 2,\mathrm 3,\mathrm 4)$, and $rbmo(\mu)$ of Yang are equivalent with each other. Then the spaces $rbmo^{1}_{\mathscr Q}(\mu)\textendash rbmo^{4}_{\mathscr Q}(\mu)$ inherit all the properties of $rbmo(\mu)$. Therefore, by \cite[Proposition 2.2]{5} and Theorem \ref{t3.x1}, we have the following Proposition \ref{2.9}.
\begin{proposition}\label{2.9}
Let $i= \mathrm 1,\mathrm 2,\mathrm 3,\mathrm 4$.
\begin{itemize}
\item[\rm(i)]  $ rbmo^{i}_{\mathscr Q} (\mu)$ is a Banach space of functions (modulo
additive constants).
\item[\rm(ii)] $L^{\infty}(\mu) \subset rbmo^{i}_{\mathscr Q}(\mu) \subset RBMO(\mu)$  with $C\|f\|_{RBMO(\mu)} \le \|f\|_{\ast i} \le C\|f\|_{L^{\infty}{(\mu})}$.
\item[\rm(iii)] If $ f \in rbmo^{i}_{\mathscr Q}(\mu)$, then $|f|\in rbmo^{i}_{\mathscr Q}(\mu)$ and $\left \| |f| \right \| _{\ast i} \le C \|f\|_{\ast i} $.
\item[\rm(iv)] If $f,g\in rbmo^{i}_{\mathscr Q}(\mu)$, then $\min(f,g),\max(f,g) \in rbmo^{i}_{\mathscr Q} (\mu)$ and
$$ ||\min(f,g)||_{\ast i},||\max(f,g)||_{\ast i}\le C(\|f\|_{\ast i}+||g||_{\ast i}).$$
\end{itemize}
\end{proposition}

\noindent Shining Li,  Haijing Zhao and Baode Li (Corresponding author)
\medskip

\noindent College of Mathematics and System Sciences\\
 Xinjiang University\\
 Urumqi, 830017\\
P. R. China
\smallskip

\noindent{E-mail }:\\
\texttt{shiningli1998@qq.com} (Shining Li)\\
\texttt{1845423133@qq.com} (Haijing Zhao)\\
\texttt{baodeli@xju.edu.cn} (Baode Li)\\

\bigskip\medskip

\end{document}